\documentclass[11pt]{article}

\usepackage[T1]{fontenc}
\usepackage{amsmath,amsthm,amssymb,mathtools}
\usepackage{newtxtext,newtxmath}
\usepackage{microtype}
\usepackage[margin=1in]{geometry}
\usepackage{hyperref}
\usepackage{bookmark}

\allowdisplaybreaks
\newtheorem{theorem}{Theorem}[section]
\newtheorem{proposition}[theorem]{Proposition}
\newtheorem{lemma}[theorem]{Lemma}
\newtheorem{corollary}[theorem]{Corollary}
\theoremstyle{definition}
\newtheorem{example}[theorem]{Example}
\newcommand{\Nop}{\mathcal N}
\newcommand{\Schubert}{\mathfrak S}
\newcommand{\R}{\mathbb R}
\newcommand{\Znonneg}{\mathbb Z_{\ge0}}
\newcommand{\supp}{\operatorname{supp}}

\title{Normalized skew Schur polynomials are Lorentzian}
\author{Philip B. Zhang\\[2pt]
  \small College of Mathematical Sciences \& Institute of Mathematics and\\
  \small Interdisciplinary Sciences, Tianjin Normal University,\\
  \small Tianjin 300387, P. R. China\\[2pt]
  \small \texttt{zhang@tjnu.edu.cn}}
\date{}

\hypersetup{
  colorlinks=true,
  linkcolor=blue,
  citecolor=blue,
  pdfauthor={Philip B. Zhang},
  pdftitle={Normalized skew Schur polynomials are Lorentzian},
  pdfsubject={Skew Schur polynomials, Schubert polynomials, and Lorentzian polynomials},
  pdfkeywords={skew Schur polynomials, Lorentzian polynomials, dually Lorentzian polynomials, rectangular complementation, skew Kostka numbers}
}

\begin{document}
\maketitle

\begin{abstract}
We prove the conjecture of Huh, Matherne, M\'esz\'aros, and St.~Dizier that the normalization of every skew Schur polynomial in finitely many variables is Lorentzian.  We first realize every nonzero skew Schur polynomial in finitely many variables as a specialization of a Schubert polynomial and prove that it is dually Lorentzian.  The dual Jacobi--Trudi identity then identifies its normalization with the finite dual of a skew Schur polynomial obtained by rectangular complementation.  As a consequence, skew Kostka numbers satisfy log-concavity inequalities along the root directions.
\end{abstract}



\section{Introduction}

Br\"and\'en and Huh \cite{BH} introduced Lorentzian polynomials.  Huh,
Matherne, M\'esz\'aros, and St.~Dizier
\cite[Theorem~3 and Conjecture~19]{HMMSD} proved that normalized Schur
polynomials are Lorentzian and conjectured the corresponding result for skew
Schur polynomials.  In this paper, we prove their conjecture.

For $n\ge1$, write $x=(x_1,\ldots,x_n)$ and
$x^{-1}=(x_1^{-1},\ldots,x_n^{-1})$.  For
$\alpha=(\alpha_1,\ldots,\alpha_n)\in\Znonneg^n$, set
$|\alpha|=\alpha_1+\cdots+\alpha_n$,
$x^\alpha=x_1^{\alpha_1}\cdots x_n^{\alpha_n}$, and
$\alpha!=\alpha_1!\cdots\alpha_n!$.  If
$f=\sum_{\alpha\in\Znonneg^n}c_\alpha x^\alpha$, the normalization operator
is defined by
\begin{equation}\label{eq:normalization}
  \Nop(f)
  =\sum_{\alpha\in\Znonneg^n}\frac{c_\alpha}{\alpha!}x^\alpha.
\end{equation}
For a polynomial $f$, let
$\supp(f)=\{\alpha\in\Znonneg^n:c_\alpha\ne0\}$.  Let
$\mathbf e_1,\ldots,\mathbf e_n$ be the standard basis vectors of
$\mathbb Z^n$.  A set $J\subseteq\Znonneg^n$ is \emph{$M$-convex} if, whenever
$\alpha,\beta\in J$ and $\alpha_i>\beta_i$, there is an index $j$ such that
$\alpha_j<\beta_j$ and both
$\alpha-\mathbf e_i+\mathbf e_j$ and
$\beta-\mathbf e_j+\mathbf e_i$ belong to $J$.
For background on discrete convex analysis and $M$-convexity, see Murota
\cite[Chapter~4]{Murota}.

Br\"and\'en and Huh
\cite[Definition~2.6 and Theorem~2.25]{BH} give the following characterization,
which we use as the definition of a Lorentzian polynomial.  Let
$f\in\R[x_1,\ldots,x_n]$ be homogeneous of degree $d$ with nonnegative
coefficients.  If $d=0$ or $1$, then $f$ is Lorentzian.  If $d\ge2$, then $f$
is Lorentzian precisely when $\supp(f)$ is $M$-convex and, for every
$\gamma\in\Znonneg^n$ with $|\gamma|=d-2$, the Hessian of
$\partial^\gamma f$ has at most one positive eigenvalue, where
$\partial^\gamma=\partial_{x_1}^{\gamma_1}\cdots
\partial_{x_n}^{\gamma_n}$.  This convention includes the zero polynomial.
Throughout the paper, $n$ denotes the number of variables, and every
Lorentzian or dually Lorentzian statement refers to the displayed polynomial
ring.  Our main result is as follows.

\begin{theorem}\label{thm:main}
Let $\mu\subseteq\lambda$ be partitions and let $n\ge1$.  Then
$\Nop\bigl(s_{\lambda/\mu}(x_1,\ldots,x_n)\bigr)$ is Lorentzian.
\end{theorem}

The Littlewood--Richardson rule (see Macdonald
\cite[Chapter~I]{Macdonald}) expresses a skew Schur polynomial as a
nonnegative integral linear combination of Schur polynomials.  This expansion
does not reduce Theorem~\ref{thm:main} to the result for ordinary Schur
polynomials because Lorentzian polynomials are not preserved by arbitrary
nonnegative linear combinations.

Our proof uses the following intermediate result.

\begin{theorem}\label{thm:dual-intro}
Let $D$ be a skew shape and let $n\ge1$.  If
$s_D(x_1,\ldots,x_n)\ne0$, then $s_D(x_1,\ldots,x_n)$ is dually
Lorentzian.
\end{theorem}

Let $f\in\R[x_1,\ldots,x_n]$ be homogeneous.  Suppose that
$\kappa=(\kappa_1,\ldots,\kappa_n)\in\Znonneg^n$ bounds the exponent of each
variable in every monomial of $f$.  Write
$f^{\vee,\kappa}:=\Nop\bigl(x^\kappa
f(x_1^{-1},\ldots,x_n^{-1})\bigr)$.  In the terminology of Ross, S\"u\ss,
and Wannerer \cite[Definition~4.2]{RSW}, $f$ is dually Lorentzian when
$f^{\vee,\kappa}$ is Lorentzian.  Ross, S\"u\ss, and Wannerer
\cite[Example~4.8]{RSW} give a dually Lorentzian polynomial whose
normalization is not Lorentzian.  Thus
Theorem~\ref{thm:dual-intro} alone does not prove Theorem~\ref{thm:main}.

Billey, Jockusch, and Stanley
\cite[Proposition~2.2, Theorem~2.2, and the proof of
Corollary~2.4]{BJS} associate with a skew shape $D$ a
$321$-avoiding permutation whose Schubert polynomial is a flagged skew Schur
polynomial.  After sufficiently many fixed points are added, their formula
specializes to $s_D(x_1,\ldots,x_n)$ when the extra variables are set equal
to zero.  Huh, Matherne,
M\'esz\'aros, and St.~Dizier \cite[Theorem~6]{HMMSD} proved that the finite
dual of every Schubert polynomial is Lorentzian.  Ross, S\"u\ss, and Wannerer
\cite[Theorems~4.5 and~5.12]{RSW} expressed this result by saying that
Schubert polynomials are dually Lorentzian and proved that this property is
preserved under nonnegative linear substitutions.  These results yield
Theorem~\ref{thm:dual-intro}.

To pass from Theorem~\ref{thm:dual-intro} to Theorem~\ref{thm:main}, we use
the rectangular complement
$C_{n,R}(\lambda/\mu)=(R^n,\mu)/\lambda$, where $R\ge\lambda_1$.  The dual
Jacobi--Trudi identity identifies
$\Nop(s_{\lambda/\mu}(x_1,\ldots,x_n))$ with the finite dual of
$s_{C_{n,R}(\lambda/\mu)}(x_1,\ldots,x_n)$.  The latter skew Schur polynomial
is dually Lorentzian by Theorem~\ref{thm:dual-intro}, and hence its finite
dual is Lorentzian.  For related work on complements of Schubert polynomials,
see Fan, Guo, and Liu \cite{FGL}.

The remainder of the paper is organized as follows.
Section~\ref{sec:dually-lorentzian} recalls dually Lorentzian polynomials,
realizes skew Schur polynomials as specializations of Schubert polynomials,
and proves Theorem~\ref{thm:dual-intro}.
Section~\ref{sec:rectangular-complement} establishes the rectangular
complement identities and proves Theorem~\ref{thm:main}.  It also records the
dependence of the complement on the chosen rectangle and derives the
log-concavity inequalities for skew Kostka numbers.

\section{Dually Lorentzian skew Schur polynomials}
\label{sec:dually-lorentzian}

For $\alpha,\kappa\in\Znonneg^n$, the notation $\alpha\le\kappa$ means
coordinatewise inequality.  Suppose that $\kappa\in\Znonneg^n$ and
$f\in\R[x_1,\ldots,x_n]$ satisfy $\alpha\le\kappa$ for every
$\alpha\in\supp(f)$.  Following Ross, S\"u\ss, and Wannerer
\cite[Definition~4.2 and Remark~4.3]{RSW}, define
\begin{equation}\label{eq:dual-definition}
  f^{\vee,\kappa}
  :=\Nop\bigl(x^\kappa f(x^{-1})\bigr).
\end{equation}
Equivalently, if $f=\sum_{\alpha\le\kappa}c_\alpha x^\alpha$, then
\[
  f^{\vee,\kappa}
  =\sum_{\alpha\le\kappa}
   \frac{c_\alpha}{(\kappa-\alpha)!}x^{\kappa-\alpha}.
\]

A homogeneous polynomial $f$ is \emph{dually Lorentzian} if
$f^{\vee,\kappa}$ is Lorentzian for some, and hence every, vector $\kappa$
satisfying this support bound.

Schubert polynomials were introduced by Lascoux and Sch\"utzenberger
\cite{LS}.  Their Newton polytopes were studied by Monical, Tokcan, and Yong
\cite{MTY} and by Fink, M\'esz\'aros, and St.~Dizier \cite{FMSD}.  We recall
the two results on dually Lorentzian polynomials needed below.  Write
$\Schubert_w$ for the Schubert polynomial indexed by $w$.  Huh, Matherne,
M\'esz\'aros, and St.~Dizier \cite[Theorem~6]{HMMSD} proved that, for
$w\in S_q$, the finite dual of $\Schubert_w$ with respect to
$(q-1,\ldots,q-1)$ is Lorentzian.  Ross, S\"u\ss, and Wannerer
\cite[Theorem~4.5]{RSW} restated this result in the following form.

\begin{theorem}\label{thm:schubert-dual}
Let $w$ be a permutation.  Then the Schubert polynomial $\Schubert_w$ is
dually Lorentzian.
\end{theorem}

Ross, S\"u\ss, and Wannerer \cite[Theorem~5.12]{RSW} further proved that the
dually Lorentzian property is preserved under nonnegative linear
substitutions.  More precisely, if
$f\in\R[x_1,\ldots,x_N]$ is dually Lorentzian,
$A\in\R_{\ge0}^{N\times m}$, and
$\mathbf y=(y_1,\ldots,y_m)^{\mathsf T}$, then
$f(A\mathbf y)$ is dually Lorentzian.

We next realize a skew Schur polynomial in finitely many variables as a
specialization of a Schubert polynomial.  For background on skew Schur
polynomials, see Stanley \cite[Chapter~7]{StanleyEC2} and Macdonald
\cite[Chapter~I]{Macdonald}.  We draw partitions and skew diagrams in English
notation and extend partitions by zero parts whenever necessary.  Rows are
numbered from top to bottom and columns from left to right, so a box $(i,j)$
lies in row $i$ and column $j$.  A
semistandard tableau is weakly increasing along rows and strictly increasing
down columns.  If $m_i(T)$ is the number of entries equal to $i$ in a
tableau $T$, its weight is $x^T=x_1^{m_1(T)}\cdots x_n^{m_n(T)}$.  The skew
Schur polynomial is $s_D(x_1,\ldots,x_n)=\sum_T x^T$, where the sum ranges
over all semistandard tableaux of shape $D$ with entries in
$\{1,\ldots,n\}$.  Permutations are written in one-line notation.  Billey,
Jockusch, and Stanley \cite[Proposition~2.1]{BJS} associate a skew shape
$\Sigma_w$ with every $321$-avoiding permutation $w$.  Their converse
construction \cite[Proposition~2.2]{BJS} will be used in the following lemma.

\begin{lemma}\label{lem:bjs-shape}
Let $D$ be a nonempty skew shape.  Then there is a $321$-avoiding permutation
$w$ such that
$s_{\Sigma_w}(x_1,\ldots,x_n)=s_D(x_1,\ldots,x_n)$ for every integer $n\ge1$.
\end{lemma}

\begin{proof}
Write $D=\lambda/\mu$.  Translate $D$ far enough downward that every box
$(i,j)$ of the resulting diagram $D'$ satisfies $i>j$.  Translation
preserves horizontal and vertical adjacencies.  For every $n\ge1$, it
therefore induces a weight-preserving bijection between the semistandard
tableaux of $D$ and $D'$ with entries in $\{1,\ldots,n\}$.  Regard $D'$ as a
poset with the product order and label each box $(i,j)\in D'$ by $i-j$.

We claim that $D'$ is convex in $\mathbb Z^2$ with the product order.
Before translation, let $(i,j),(k,\ell)\in D$ and $(a,b)\in\mathbb Z^2$
satisfy $(i,j)\le(a,b)\le(k,\ell)$.  Then
\[
  b\le\ell\le\lambda_k\le\lambda_a,
  \qquad
  b\ge j>\mu_i\ge\mu_a,
\]
and hence $(a,b)\in D$.  Translation is an automorphism of the product
order, so the same conclusion holds for $D'$.

Billey, Jockusch, and Stanley \cite[Proposition~2.2]{BJS} show that there is
a $321$-avoiding permutation $w$ whose labeled poset $(P_w,\omega_w)$ is
isomorphic to this labeled poset.  Billey, Jockusch, and Stanley
\cite[Proposition~2.1 and the discussion following it]{BJS} identify
$(P_w,\omega_w)$ with the associated skew shape $\Sigma_w$, embedded as a
convex subset of $\mathbb Z^2$, in such a way that
$\omega_w(i,j)=i-j$.  We use this embedding in English notation.  Let
$\varphi:D'\longrightarrow\Sigma_w$ be the resulting label-preserving poset
isomorphism.

By convexity, every cover in $D'$ or $\Sigma_w$ is a unit horizontal or
vertical step.  Along a horizontal cover, the label $i-j$ decreases by one,
whereas along a vertical cover, it increases by one.  Since $\varphi$ preserves
labels, it sends horizontal covers to horizontal covers and vertical covers to
vertical covers.

For a filling $T$ of $D'$, define the filling $\varphi_*T$ of $\Sigma_w$ by
$(\varphi_*T)(\varphi(b))=T(b)$ for every box $b\in D'$.  A filling is
semistandard precisely when it is weakly increasing along horizontal covers
and strictly increasing along vertical covers.  It follows that $T$ is
semistandard if and only if $\varphi_*T$ is semistandard.  These two fillings
have the same weight.  Together with the downward translation, this gives a
weight-preserving bijection between the semistandard tableaux of $D$ and
$\Sigma_w$ with entries in $\{1,\ldots,n\}$ for every $n\ge1$.  Consequently,
$s_{\Sigma_w}(x_1,\ldots,x_n)=s_D(x_1,\ldots,x_n)$ for every $n\ge1$.
\end{proof}

For $u=(u_1,\ldots,u_a)\in S_a$ and
$v=(v_1,\ldots,v_b)\in S_b$, define
\[
  u\times v
  =(u_1,\ldots,u_a,v_1+a,\ldots,v_b+a)\in S_{a+b}.
\]
Let $\mathrm{id}_m$ denote the identity permutation in $S_m$.  For Schubert
polynomials, we use the convention of Billey, Jockusch, and Stanley
\cite[p.~346]{BJS}.

\begin{proposition}\label{prop:schubert}
Let $D$ be a skew shape and let $n\ge1$.  There exist integers
$p,m_0\ge1$ and a $321$-avoiding permutation $w\in S_p$ such that, for
every $m\ge m_0$,
\[
  s_D(x_1,\ldots,x_n)
  =\Schubert_{\mathrm{id}_m\times w}
   (x_1,\ldots,x_n,0,\ldots,0).
\]
Here the variables of the Schubert polynomial are
$z_1,\ldots,z_{m+p-1}$, with $z_i=x_i$ for $1\le i\le n$ and $z_i=0$
for $n<i\le m+p-1$.
\end{proposition}

\begin{proof}
If $D$ is empty, take $p=1$, $w=\mathrm{id}_1$, and $m_0=n$.  Both
sides are then equal to $1$.  Assume that $D$ is nonempty, and choose
$w\in S_p$ and $\Sigma_w$ as in Lemma~\ref{lem:bjs-shape}.

For $q\ge1$, put $X_q=(z_1,\ldots,z_q)$.  The code of
$w=(w_1,\ldots,w_p)$ is $c(w)=(c_1,\ldots,c_p)$, where
\[
  c_i=\#\{j:i<j\le p\text{ and }w_j<w_i\}.
\]
Following Billey, Jockusch, and Stanley \cite[p.~364]{BJS}, let
$\widehat\phi=(\widehat\phi_1,\ldots,\widehat\phi_\ell)$ be the flag attached
to $w$.  Its entries are the positions of the nonzero entries of $c(w)$.  The
expression $s_{\Sigma_w}(X_{\widehat\phi_1},\ldots,
X_{\widehat\phi_\ell})$ restricts the entries in row $i$ to be at most
$\widehat\phi_i$.  By \cite[Theorem~2.2]{BJS},
\[
  \Schubert_w
  =s_{\Sigma_w}\bigl(
      X_{\widehat\phi_1},\ldots,X_{\widehat\phi_\ell}
    \bigr).
\]
Moreover, for every integer $m\ge1$, the proof of
\cite[Corollary~2.4]{BJS} shows that the skew shape associated with
$\mathrm{id}_m\times w$ is a horizontal translate of $\Sigma_w$, while its flag is
$(\widehat\phi_1+m,\ldots,\widehat\phi_\ell+m)$.
Since translation does not change the tableau generating function, we obtain
\begin{equation}\label{eq:bjs}
  \Schubert_{\mathrm{id}_m\times w}
  =s_{\Sigma_w}\bigl(
      X_{\widehat\phi_1+m},\ldots,
      X_{\widehat\phi_\ell+m}
    \bigr).
\end{equation}

Since the last component of the code of $w\in S_p$ is zero,
\[
  \widehat\phi_i\le p-1
  \qquad (1\le i\le\ell).
\]
Thus only $z_1,\ldots,z_{m+p-1}$ occur in \eqref{eq:bjs}.  Choose
$m_0$ so that, for $m\ge m_0$,
\[
  \widehat\phi_i+m\ge n\quad(1\le i\le\ell),
  \qquad
  m+p-1\ge n.
\]
After setting the variables as in the statement, the surviving tableaux are
precisely those with entries in $\{1,\ldots,n\}$.  Equation~\eqref{eq:bjs} and
Lemma~\ref{lem:bjs-shape} now give
\[
  \Schubert_{\mathrm{id}_m\times w}
   (x_1,\ldots,x_n,0,\ldots,0)
  =s_{\Sigma_w}(x_1,\ldots,x_n)
  =s_D(x_1,\ldots,x_n).
\]
This proves the proposition.
\end{proof}

\begin{proof}[Proof of Theorem~\ref{thm:dual-intro}]
Choose $p,m_0,w$ as in Proposition~\ref{prop:schubert}, fix
$m\ge m_0$, and put
\[
  N=m+p-1,
  \qquad
  v=\mathrm{id}_m\times w\in S_{N+1}.
\]
The choice of $m_0$ ensures that $N\ge n$.
Regard the Schubert polynomial of $v$ as
\[
  F(z_1,\ldots,z_{N+1})
  =\Schubert_v(z_1,\ldots,z_{N+1}).
\]
By Theorem~\ref{thm:schubert-dual}, $F$ is dually Lorentzian.

Let $\mathbf x=(x_1,\ldots,x_n)^{\mathsf T}$ and set
\[
  A=
  \begin{pmatrix}
    I_n\\
    0_{(N+1-n)\times n}
  \end{pmatrix}
  \in\R_{\ge0}^{(N+1)\times n}.
\]
Ross, S\"u\ss, and Wannerer \cite[Theorem~5.12]{RSW} show that
$F(A\mathbf x)$ is dually Lorentzian.  A Schubert polynomial indexed by an
element of $S_{N+1}$ does not depend on its last variable.  Hence
Proposition~\ref{prop:schubert} gives
\[
  F(A\mathbf x)
  =\Schubert_{\mathrm{id}_m\times w}
    (x_1,\ldots,x_n,0,\ldots,0)
  =s_D(x_1,\ldots,x_n),
\]
which proves the theorem.
\end{proof}

In the preceding proof, the substitution only sets variables equal to zero.
The exponent bound used in the next section will instead be read directly
from the columns of the complementary skew shape.

\section{Rectangular complementation and the main theorem}
\label{sec:rectangular-complement}

Throughout this section, a skew Schur polynomial without an explicit
alphabet is understood to be specialized to $x_1,\ldots,x_n$.
For a partition $\nu$, write $\nu'$ for its conjugate.  For a skew shape $D$,
write $|D|$ for its number of boxes.

For integers $n,R\ge1$, let
$R^n=(R,\ldots,R)$ denote the rectangular partition with $n$ parts.  If $\mu$
is a partition with $\mu_1\le R$, write
$(R^n,\mu)$ for the concatenated partition
$(R,\ldots,R,\mu_1,\mu_2,\ldots)$.  Whenever
$\lambda\subseteq(R^n,\mu)$, write
$C_{n,R}(\lambda/\mu)=(R^n,\mu)/\lambda$ for the rectangular complement.

We begin with the identity needed in the proof of the main theorem.

\begin{proposition}\label{prop:complement}
Let $\mu\subseteq\lambda$ be partitions with $\lambda\ne\varnothing$, let
$n\ge1$, and suppose that
$s_{\lambda/\mu}(x_1,\ldots,x_n)\ne0$.  Then, for every integer
$R\ge\lambda_1$, the containment $\lambda\subseteq(R^n,\mu)$ holds.  Hence
$C_{n,R}(\lambda/\mu)$ is a well-defined skew shape.  Moreover,
\[
  |C_{n,R}(\lambda/\mu)|=nR-|\lambda/\mu|,
\]
and
\begin{equation}\label{eq:complement-reverse}
  (x_1\cdots x_n)^R
  s_{C_{n,R}(\lambda/\mu)}(x_1^{-1},\ldots,x_n^{-1})
  =s_{\lambda/\mu}(x_1,\ldots,x_n).
\end{equation}
\end{proposition}

\begin{proof}
Put $r=\lambda_1$.  Since the specialization is nonzero, there is a
semistandard tableau of shape $\lambda/\mu$ with entries in
$\{1,\ldots,n\}$.  Strict increase down columns therefore gives
\[
  0\le\lambda'_j-\mu'_j\le n
  \qquad (1\le j\le r).
\]
Every part of $\mu$ is at most $\lambda_1\le R$, so $(R^n,\mu)$ is a
partition.  After padding $\lambda'$ and $\mu'$ with zeros to length $R$, the
column inequalities give
\[
  \lambda'_j\le n+\mu'_j
  \qquad (1\le j\le R),
\]
where the inequality is trivial for $j>r$.  Hence
$\lambda\subseteq(R^n,\mu)$, so $C_{n,R}(\lambda/\mu)$ is a well-defined skew
shape.

Let $e_a(x)$ denote the elementary symmetric polynomial of degree $a$ in the
alphabet $x$, with $e_0=1$ and $e_a=0$ for $a<0$ or $a>n$.  Applying the dual
Jacobi--Trudi identity to the alphabet $x^{-1}$ (see
Macdonald \cite[Chapter~I, (5.5)]{Macdonald}), we have
\[
  s_{\lambda/\mu}(x^{-1})
  =\det\bigl(
      e_{\lambda'_i-\mu'_j-i+j}(x^{-1})
    \bigr)_{1\le i,j\le r}.
\]
The determinant may be enlarged to size $R$.  Indeed, after padding
$\lambda'$ and $\mu'$ with zeros, the entries with $i>r$ and $j\le r$ vanish,
while the lower-right block is
\[
  \bigl(e_{-i+j}(x^{-1})\bigr)_{r<i,j\le R},
\]
which is upper triangular with diagonal entries $e_0=1$.  Hence
\[
  s_{\lambda/\mu}(x^{-1})
  =\det\bigl(
      e_{\lambda'_i-\mu'_j-i+j}(x^{-1})
    \bigr)_{1\le i,j\le R}.
\]
In $n$ variables,
\[
  (x_1\cdots x_n)e_a(x^{-1})=e_{n-a}(x)
\]
for every integer $a$.
Multiplying each row of the determinant by $x_1\cdots x_n$ and then
transposing gives
\begin{align*}
  (x_1\cdots x_n)^R s_{\lambda/\mu}(x^{-1})
  &=\det\bigl(
      e_{n+\mu'_i-\lambda'_j-i+j}(x)
    \bigr)_{1\le i,j\le R}\\
  &=s_{(R^n,\mu)/\lambda}(x),
\end{align*}
since
\[
  (R^n,\mu)'_i=n+\mu'_i
  \qquad (1\le i\le R).
\]
This gives the forward complement identity.  Also,
\[
  |C_{n,R}(\lambda/\mu)|
  =nR+|\mu|-|\lambda|
  =nR-|\lambda/\mu|.
\]
Finally, replace $x$ by $x^{-1}$ in the forward complement identity.  As an
identity of Laurent polynomials, this gives
\[
  (x_1\cdots x_n)^{-R}s_{\lambda/\mu}(x)
  =s_{C_{n,R}(\lambda/\mu)}(x^{-1}).
\]
Multiplying by $(x_1\cdots x_n)^R$ proves
\eqref{eq:complement-reverse}.
\end{proof}

Both the nonvanishing assumption and the bound $R\ge\lambda_1$ are needed.
For
$\lambda=(1,1)$, $\mu=\varnothing$, and $n=R=1$, the specialization vanishes
and $(R^n,\mu)$ does not contain $\lambda$.  For $\lambda=(2)$,
$\mu=\varnothing$, and $n=R=1$, the specialization is nonzero, but the same
containment fails because $R<\lambda_1$.

\par\medskip
\begin{corollary}\label{cor:normalization}
Under the hypotheses of Proposition~\ref{prop:complement}, for every
$R\ge\lambda_1$ let
\[
  E_R=C_{n,R}(\lambda/\mu),
  \qquad
  \kappa_R=(R,\ldots,R)\in\Znonneg^n.
\]
Then $s_{E_R}\ne0$, every $\alpha\in\supp(s_{E_R})$ satisfies
$\alpha\le\kappa_R$, and
\begin{equation}\label{eq:normalization-identity}
  \Nop(s_{\lambda/\mu})=s_{E_R}^{\vee,\kappa_R}.
\end{equation}
\end{corollary}

\begin{proof}
Equation~\eqref{eq:complement-reverse} implies that $s_{E_R}\ne0$.  The shape
$E_R$ has at most $R$ nonempty columns.  Since a fixed entry occurs at most
once in each column of a semistandard tableau, every
$\alpha\in\supp(s_{E_R})$ satisfies
\[
  \alpha\le\kappa_R.
\]
The reverse identity also gives
$x^{\kappa_R}s_{E_R}(x^{-1})=s_{\lambda/\mu}(x)$.  Applying $\Nop$ and using
\eqref{eq:dual-definition} proves \eqref{eq:normalization-identity}.
\end{proof}

\begin{example}\label{ex:complement}
Let $\lambda=(3,2)$, $\mu=(1)$, and $n=3$.  Taking the minimal width
$R=\lambda_1=3$ gives
\[
  E_3=C_{3,3}(\lambda/\mu)=(3,3,3,1)/(3,2).
\]
The two shapes have sizes $4$ and $5=3\cdot3-4$, respectively.  Let
$m_\nu=m_\nu(x_1,x_2,x_3)$ denote the monomial symmetric polynomial indexed
by $\nu$.  Direct expansion gives
\begin{align*}
  s_{(3,2)/(1)}
  &=m_{(3,1)}+2m_{(2,2)}+3m_{(2,1,1)},\\
  s_{E_3}
  &=m_{(3,2)}+2m_{(3,1,1)}+3m_{(2,2,1)}.
\end{align*}
The two supports correspond under
$\alpha\mapsto(3,3,3)-\alpha$.  Therefore
\[
  s_{E_3}^{\vee,(3,3,3)}
  =\frac{1}{6}m_{(3,1)}
   +\frac{1}{2}m_{(2,2)}
   +\frac{3}{2}m_{(2,1,1)}
  =\Nop\bigl(s_{(3,2)/(1)}\bigr).
\]
This is the identity in Corollary~\ref{cor:normalization} for this
complement.
\end{example}

\begin{proof}[Proof of Theorem~\ref{thm:main}]
If $\lambda/\mu$ is empty, then $\Nop(s_{\lambda/\mu})=1$.  If
$s_{\lambda/\mu}(x_1,\ldots,x_n)=0$, then its normalization is $0$.
Both polynomials are Lorentzian by the definition in the Introduction.

Now suppose that the shape is nonempty and its specialization is nonzero.
Choose any $R\ge\lambda_1$, and let $E_R$ and $\kappa_R$ be as in
Corollary~\ref{cor:normalization}.  By that corollary, every
$\alpha\in\supp(s_{E_R})$ satisfies $\alpha\le\kappa_R$, and
\eqref{eq:normalization-identity} holds.  Theorem~\ref{thm:dual-intro} shows
that $s_{E_R}$ is dually Lorentzian.  Hence
$s_{E_R}^{\vee,\kappa_R}$ is Lorentzian.  By
\eqref{eq:normalization-identity}, this polynomial is precisely
$\Nop(s_{\lambda/\mu}(x_1,\ldots,x_n))$.  It follows that the latter is
Lorentzian in
$\R[x_1,\ldots,x_n]$.
\end{proof}

We next record how the complement depends on the choices involved.  Retain
the hypotheses and notation of Corollary~\ref{cor:normalization}.  For the
fixed presentation $\lambda/\mu$, if $R'\ge R\ge\lambda_1$, then applying
\eqref{eq:complement-reverse} with widths $R$ and $R'$ and replacing $x$ by
$x^{-1}$ gives
\[
  s_{C_{n,R'}(\lambda/\mu)}
  =(x_1\cdots x_n)^{R'-R}s_{C_{n,R}(\lambda/\mu)}.
\]
Thus changing the width multiplies the complementary polynomial by a power
of $x_1\cdots x_n$, while Corollary~\ref{cor:normalization} gives
\[
  s_{C_{n,R}(\lambda/\mu)}^{\vee,\kappa_R}
  =\Nop(s_{\lambda/\mu}).
\]
Applying the construction twice, with the presentation kept fixed, gives
$C_{n,R}((R^n,\mu)/\lambda)=(R^n,\lambda)/(R^n,\mu)$, which is a downward
translate of $\lambda/\mu$ by $n$ rows.

The complement also depends on the presentation and on the number of
variables.  It need not commute with specialization of the alphabet.  For
$D=(1)/\varnothing$ and $R=1$, for example,
\[
  C_{2,1}(D)=(1,1)/(1),
  \qquad
  C_{1,1}(D)=\varnothing,
  \qquad
  \left.s_{C_{2,1}(D)}(x_1,x_2)\right|_{x_2=0}=x_1
  \ne 1=s_{C_{1,1}(D)}(x_1).
\]
If $E=C_{2,1}(D)$, then likewise
$\left.s_E^{\vee,(1,1)}\right|_{x_2=0}=x_1$, whereas
$\bigl(s_E(x_1,0)\bigr)^{\vee,(1)}=1$.  Thus the alphabet size and the
multidegree bound must be fixed before taking a finite dual.

We conclude with a consequence for skew Kostka numbers.  For partitions
$\mu\subseteq\lambda$ and an integer $n\ge1$, set
$d=|\lambda/\mu|$.  For a weak composition
$\beta\in\Znonneg^n$ of $d$, let $K_{\lambda/\mu,\beta}$ denote the number
of semistandard tableaux of shape $\lambda/\mu$ and content $\beta$.
Equivalently,
\[
  K_{\lambda/\mu,\beta}
  =[x^\beta]s_{\lambda/\mu}(x_1,\ldots,x_n).
\]
Recall that $\mathbf e_1,\ldots,\mathbf e_n$ are the standard basis vectors
of $\mathbb Z^n$.

\begin{corollary}\label{cor:skew-kostka}
Let $\mu\subseteq\lambda$ be partitions, let $n\ge1$, and put
$d=|\lambda/\mu|$.  For any $\alpha\in\Znonneg^n$ with $|\alpha|=d$
and any distinct $i,j\in\{1,\ldots,n\}$ with $\alpha_i,\alpha_j\ge1$, one has
\begin{equation}\label{eq:skew-kostka-logconcavity}
  K_{\lambda/\mu,\alpha}^{\,2}
  \ge
  K_{\lambda/\mu,\alpha+\mathbf e_i-\mathbf e_j}\,
  K_{\lambda/\mu,\alpha-\mathbf e_i+\mathbf e_j}.
\end{equation}
\end{corollary}

\begin{proof}
If $s_{\lambda/\mu}(x_1,\ldots,x_n)=0$, the assertion is trivial.  Otherwise,
set
\[
  H=\Nop(s_{\lambda/\mu})
   =\sum_{|\beta|=d}
      \frac{K_{\lambda/\mu,\beta}}{\beta!}x^\beta.
\]
By Theorem~\ref{thm:main}, $H$ is Lorentzian.  Put
$\gamma=\alpha-\mathbf e_i-\mathbf e_j$.
The assumptions on $\alpha$ imply that $\gamma\in\Znonneg^n$ and
$|\gamma|=d-2$.
Br\"and\'en and Huh
\cite[Corollary~2.11]{BH} show that every partial derivative of a Lorentzian
polynomial is Lorentzian.  Thus $Q=\partial^\gamma H$ is a Lorentzian
quadratic polynomial.  The factorial normalization gives
\[
  \partial^\beta H=K_{\lambda/\mu,\beta}
  \qquad (|\beta|=d).
\]
Hence the principal $i,j$ submatrix of the Hessian of $Q$ is
\[
  \begin{pmatrix}
    K_{\lambda/\mu,\alpha+\mathbf e_i-\mathbf e_j}
      & K_{\lambda/\mu,\alpha}\\
    K_{\lambda/\mu,\alpha}
      & K_{\lambda/\mu,\alpha-\mathbf e_i+\mathbf e_j}
  \end{pmatrix}.
\]
Since $Q$ is a Lorentzian quadratic polynomial, its Hessian has at most
one positive eigenvalue. By Cauchy interlacing, the same is true for every
principal submatrix.  Since the diagonal entries displayed above are
nonnegative, this $2\times2$ principal submatrix has nonpositive determinant.
This is exactly
\eqref{eq:skew-kostka-logconcavity}.
\end{proof}

\bigskip
\noindent\textbf{Acknowledgments.}
This work was supported by the National Natural Science Foundation of China
(No.~12171362) and the Tianjin Municipal Natural Science Foundation
(No.~25JCYBJC00430).

\end{document}